\documentclass[10pt]{amsart}
\pdfoutput=1
\usepackage{amsmath, amsthm, amssymb,slashed,stmaryrd}
\usepackage{ifpdf}
\usepackage[pdftex]{graphicx}
\usepackage{tikz}
\usetikzlibrary{matrix,arrows,calc}
\usepackage{mathtools}
\usepackage[pdftex,plainpages=false,hypertexnames=false,pdfpagelabels]{hyperref}
\usepackage{tikz-cd}

\usepackage[super]{nth}

 \theoremstyle{plain}
 \newtheorem{thm}{Theorem}[section]
    \newtheorem{claim}[thm]{Claim}
 \newtheorem{cor}[thm]{Corollary}
 \newtheorem{lem}[thm]{Lemma}
 \newtheorem{prop}[thm]{Proposition}
 \newtheorem{conj}[thm]{Conjecture}
 \theoremstyle{definition}
 \newtheorem{defn}[thm]{Definition}
 \newtheorem{notation}[thm]{Notation}
 \newtheorem{ex}[thm]{Example}
 \newtheorem{constr}[thm]{Construction} 
 \newtheorem*{thm*}{Theorem}
 \theoremstyle{remark}
 \newtheorem{rmk}[thm]{Remark}
 
 \newtheorem{question}[thm]{Question}
\numberwithin{thm}{subsection} 
\def\beq{\begin{eqnarray}}
\def\eeq{\end{eqnarray}}
 \newcommand{\bp}{\begin{proof}[Proof]}
 \newcommand{\ep}{\end{proof}}

\DeclareSymbolFont{bbold}{U}{bbold}{m}{n}
\DeclareSymbolFontAlphabet{\mathbbold}{bbold}

\def\Sym{{\rm Sym}}

\def\Spec{{\rm Spec}}

\def\Jac{{\rm Jac }\,}
\def\Gal{{\rm Gal}}

\def\Prym{{\rm Prym}}
\def\codim{{\rm codim}}

\def\Sh{{\rm Sh}}
\def\qbar{{\overline{\mathbb{Q}}}}

\def\K3{{\rm K3}}

\def\Imag{{\rm Im}}
\def\GL{{\rm GL}}
\def\GSp{{\rm GSp}}
\def\Res{{\rm Res}}
\def\inv{{\rm inv}}
\def\prim{{\rm prim}}
\def\KS{{\rm KS}}
\def\Homo{{\rm Hom}}

\newcommand{\defeq}{\vcentcolon=}

\newcommand\nc{\newcommand}
\nc\mf\mathfrak
\nc\mc\mathcal
\nc\mb\mathbb

\usepackage[
backend=biber,
style=alphabetic,
sorting=nyt,
maxnames=50
]{biblatex}

\begin{document}

\title{Attractors are not algebraic}

\author{Yeuk Hay Joshua Lam and Arnav Tripathy}

\date{\today}

\begin{abstract} The Attractor Conjecture for Calabi-Yau moduli spaces predicts the algebraicity of the moduli values of certain isolated points picked out by Hodge-theoretic conditions. We provide a family of counterexamples to the Attractor Conjecture in all suitably high, odd dimensions conditional on the Zilber-Pink conjecture. 
\end{abstract}

\maketitle 
\setcounter{tocdepth}{1}
\tableofcontents

\section{Introduction}
\subsection{Statement of results}
In this paper, we study the following remarkable conjecture due to string theorists:
\begin{conj}[Moore]\label{conj:attractor}
If $X$ is an attractor Calabi-Yau 3-fold, then it is defined over $\qbar$.
\end{conj}
We  recall the definition of an attractor Calabi-Yau variety: 

\begin{defn}
For  $X$ a Calabi-Yau $d$-fold, we say that it is an \emph{attractor variety} if there is a nonzero integral cohomology class $\gamma \in H^d(X, \mb{Z})$ satisfying 
\[
\gamma \perp H^{d-1, 1},
\]
where $H^{d-1,1}\subset H^d(X, \mb{C})$ denotes the $(d-1,1)$ piece of the Hodge decomposition. 
\end{defn}
These varieties were originally introduced and studied by~Ferrara-Kallosh-Strominger for Calabi-Yau threefolds as the case most directly of interest in string theory; Calabi-Yau fourfolds were also considered shortly thereafter~\cite[Section 3.8]{mooreleshouches}. The above definition in general dimension was then given in~\cite{Brunner-Roggenkamp}, as we discuss somewhat further in \S 2. Note to build intuition that the above condition should impose $h^{d-1,1}$ conditions, where we note $h^{d-1,1} = \dim H^1(X, TX)$ as the dimension of Calabi-Yau moduli space; as such, one typically expects attractor Calabi-Yaus (for some fixed $\gamma$) to be isolated in moduli space, which is indeed the case for the examples we consider below. It is hence certainly of interest, irrespective of the physical genesis of the question, to investigate the arithmetic structure of the points picked out by this natural Hodge-theoretic condition. 


Our main result is then the following:
\begin{thm}\label{thm:main}
 Under the Zilber-Pink conjecture, the analogue of Conjecture \ref{conj:attractor} for Calabi-Yau varieties of arbitrary dimension is false. More precisely, there exist attractor Calabi-Yau varieties in all odd dimensions except 1, 3, 5 and 9 which are not defined over $\qbar$.
\end{thm}

Indeed, for the family of counterexamples we consider, we show the much stronger statement that the set of attractor points defined over $\qbar$ must be non-Zariski dense in the Calabi-Yau moduli space. As we will check that the set of attractor points is indeed dense (even in the analytic topology), these families give extremely strong counterexamples in the sense that almost all attractor points fail to be defined over $\qbar$.

Amusingly, the Calabi-Yau examples we consider are decidedly \emph{not} counterexamples in dimensions 1, 3, 5 or 9. These examples have already been well-understood in the context of flat surfaces (in the theory of Teichm\"{u}ller dynamics). Indeed, in these cases the attractors are indeed algebraic and are moreover examples of CM points on Shimura varieties.  

While we give the specific counterexamples above due to the particular techniques we bring to bear, we expect a much more general transcendence property for these attractor points.

\begin{conj} The algebraic attractor points in the moduli space of a Calabi-Yau $X$ are Zariski dense if and only if said moduli space is a Shimura variety. \end{conj}

We pause to explain the nomenclature and history of our examples, as well as to point to some related examples. The Calabi-Yau construction we use is that of a crepant resolution of an $n$-fold cyclic cover of $\mathbb{P}^{2n-3}$ branched at a suitable hyperplane collection, following~\cite{sxzminimal}. We follow these authors in citing Dolgachev for his study of the moduli spaces thereof (as attempting to answer the famous question of B. Gross on realizing ball quotients as geometric moduli spaces) as in~\cite{DvGK, DK}, terming these \emph{Dolgachev Calabi-Yaus}.

In particular, one could certainly consider variants of the construction we investigate here, such as a family of double covers of projective space now branched at some other suitable hyperplane arrangement. This latter family contains a Calabi-Yau threefold example with non-Shimura moduli~\cite{sxzmonodromy}. Following our conjecture above, we hence suggest the following case as a particularly attractive next area for investigation:

\begin{question}
Does the Attractor Conjecture hold for the family of double cover Dolgachev Calabi-Yau threefolds?
\end{question}

  \subsection{History of the problem and related  works}
Attractor varieties in the context of Calabi-Yau threefolds were originally discovered by Ferrara-Kallosh-Strominger in the context of Calabi-Yau threefold compactifications of string theory. They have been the subject of focused study since; mathematically, for example, they are conjectured to govern the behavior of the enumerative geometry of Calabi-Yau threefolds~\cite{kontsevichsoibelman}. Moore in~\cite{moore} performed an in-depth study and made various conjectures about their possible arithmetic properties, including the Conjecture~\ref{conj:attractor} above. In particular, Moore investigated various examples such as $S \times E$ for $S$ a K3 surface and $E$ an elliptic curve, a quotient thereof known as the FHSV model, and abelian threefolds. In all these cases, the attractor points are defined over $\mathbb{Q}$; however, note that all these examples have Shimura moduli (and indeed, the attractor points are special points in said Shimura variety). 

These attractor points have many other becoming properties analogous to those of special points of Shimura varieties. Douglas and his coauthors studied the distribution of these points in their moduli space in~\cite{douglas, denefdouglas, DSZ1, DSZII, DSZIII}; the last series of papers by Douglas-Shiffman-Zelditch in particular developed strong heuristics to suggest that attractor points equidistribute in moduli space with its natural Weil-Petersson metric (together with strong numerical evidence in myriad cases to support said claim). We in fact use such distributional results in our proof of Theorem~\ref{thm:main}, although we only need the much weaker statement that attractor points are Zariski dense, which we can verify directly in \S \ref{subsection:dense}.

\subsection{Outline of proof of Theorem \ref{thm:main}}
We sketch the proof of Theorem \ref{thm:main}.
 We proceed by contradiction, so we assume that all attractor Calabi-Yau varieties are in fact defined over $\qbar$. We consider the Dolgachev Calabi-Yau varieties as constructed in \S \ref{subsection:dolgcys}; these give examples of Calabi-Yau varieties defined in all odd dimensions and will provide our counterexamples in almost all cases. These Dolgachev Calabi-Yau varieties $X$ are constructed  from an associated curve $C$, and the middle Hodge structures of $X$ and $C$ are closely related as reviewed in Section \ref{subsection:dolgcys}.
   Next, it is not difficult to check that the attractors are Zariski (in fact, even analytically) dense in the moduli space $\mc{M}$.
    Secondly we show that, for $X$ a Dolgachev Calabi-Yau variety, if it is attractor and defined over $\qbar$, then its Jacobian splits in the isogeny category as $A_1\times A_2$ where the abelian variety $A_1$ has complex multiplication (CM) by a fixed cyclotomic field. The crucial ingredient here is a theorem of Shiga-Wolfart (following W\"{u}stholtz) in transcendence theory, which, informally, implies that an abelian variety defined over $\qbar$ is CM as soon as it has sufficiently many algebraic period ratios, which in our case follows from the attractor condition and prior Hodge-theoretic analysis.

   This splitting of $\Jac(C)$ up to isogeny then may now be thought of as a problem in the intersection theory of Shimura varieties: \emph{a priori}, $\Jac(C)$ naturally defines a point of an ambient Shimura variety $\Sh$ and the isogeny splitting condition above implies that $\mc{M}$ intersects the Hecke translates of a sub-Shimura variety $\Sh_A$ of $\Sh$ in a dense set of points. The attractor condition has hence reduced to a problem in \emph{unlikely intersection theory}. In particular, when the codimensions of $\mc{M}$ and $\Sh_A$ in $\Sh$ sum to less than the dimension of $\Sh$, the Zilber-Pink conjecture implies that $\mc{M}$  is contained in some proper Shimura subvariety. This is the point where the argument fails for small values of the dimension of the Calabi-Yau varieties; in fact, in these cases the moduli space $\mc{M}$ turns out to be a Shimura variety, the attractor points are CM points, and therefore the Attractor Conjecture holds. In the general case, we instead use a result of Deligne-Mostow on the monodromy groups of these varieties to show that, for almost all dimensions, $\mc{M}$ not contained in any proper Shimura subvariety of $\Sh$, and hence we have a contradiction as desired.

\subsection{Outline of the rest of the paper}
     In Section \ref{section:attconj} we introduce the attractor condition on a Calabi-Yau variety as well as the Attractor Conjecture, along with providing somewhat more context for the interested reader. Section $3$ continues with the main thrust of the proof above by defining the Dolgachev Calabi-Yaus $X$ along with their associated curves $C$ before establishing the relation between the Hodge structures thereof. Section 4 applies the theorem of Shiga-Wolfart to reduce to a problem in Shimura theory before setting up the formalism of the ambient Shimura variety $\Sh$ and its special Shimura subvariety $\Sh_A$. We finally conclude in Section 5 with a discussion on the unlikely intersection theory of this Shimura variety problem.
\subsection{Notations and conventions}
We set a few conventions now. We work throughout with the Hermitian intersection pairing on the middle-degree complex cohomology of a manifold; under this pairing, distinct Hodge summands are orthogonal. For a vector space $V$ defined over some field $K$ and a field extension $K \subset L$, we often denote the extension of scalars $V \otimes_K L$ as $V_L.$ For an algebraic variety $X$ over $\mb{C}$, we will sometimes say that $X$ is \textit{algebraic} to mean that it is defined over $\overline{\mb{Q}}$, since the latter is tautologically  equivalent to the statement that the coordinates of the point corresponding to $X$ in its moduli space being algebraic numbers.

Much of the motivation for this work arose from discussions with Shamit Kachru and Akshay Venkatesh, and it is a pleasure to thank them. We are also grateful to Matt Emerton, Phil Engel, Mark Kisin, Barry Mazur, Greg Moore, Curt McMullen, Minhyong Kim, Ananth Shankar, and Max Zimet for illuminating comments, questions, and discussions. AT is supported under NSF MSPRF grant 1705008.

\section{The Attractor Conjecture}\label{section:attconj}

We recall the attractor condition for (higher-dimensional) Calabi-Yaus with slightly more precision now. Note that in general, we take a Calabi-Yau variety to be a smooth, projective variety $X$ with trivial canonical bundle and (\emph{a priori}) defined over the complex numbers $\mathbb{C}$. In practice, however, we will work with a specific family defined shortly in \S \ref{subsection:dolgcys}. 

\begin{defn}\label{defn:attractor}
Given a Calabi-Yau $d$-fold $X$, then for each nonzero class $\gamma \in H^d(X, \mb{Z})$, $X$ is said to be an \emph{attractor} for the class $\gamma$ if the following condition
\[
\gamma \perp H^{1, d-1}(X)
\]
holds. If in addition we have $\gamma^{d, 0} \ne 0$, we will refer to $X$ as an \emph{attractor point}. \end{defn}
\begin{rmk} Note the condition $\gamma^{d, 0} \ne 0$ is vacuous in the original case of threefolds, as $\gamma^{3, 0} = 0$ and the attractor condition would simply imply $\gamma = 0$ by Hodge theory. Our reason for emphasizing this condition is that in general, for higher-dimensional Calabi-Yaus, omitting this condition would allow the possibility for non-isolated attractors (i.e. positive-dimensional families); it is indeed straightforward to construct examples where this happens, such as the very families of Dolgachev Calabi-Yaus considered in this paper when $n$ is not prime. As such, the higher-dimensional Attractor Conjecture would be trivially false. \end{rmk}

In fact, this attractor condition in high dimensions is essentially due to~\cite{Brunner-Roggenkamp} in their \S 4, although they instead consider critical points of the central charge function $$Z_{\gamma}\colon= \frac{\langle \gamma, \Omega \rangle}{\sqrt{\langle \Omega, \overline{\Omega}\rangle}}.$$ One may easily verify that this condition is equivalent to the attractor condition; it may be amusing to note that in fact there is a natural flow on (the universal cover) of the Calabi-Yau moduli space induced by gradient flow for $\log |Z_{\gamma}|^2$ which naturally dynamically produces these special points as fixed points of said flow. These considerations will play no role in our analysis, however. 

\begin{conj}[{\cite[Conjecture 8.2.2]{moore}} for  the case of threefolds]\label{conj:mooreattractor}
If $X$ is an attractor variety for a non-zero class $\gamma \in H^d(X, \mb{Q})$ such that $\gamma^{d,0}\neq 0$, then it has a model over $\overline{\mb{Q}}$.
\end{conj}

The above conjecture would suggest that these points picked out by the Hodge-theoretic attractor condition are a class of special points analogous in myriad aspects to special points of Shimura varieties. (Indeed, Moore also makes a counterpart conjecture that the periods of these points are algebraic.) And while Moore originally makes the conjecture for Calabi-Yau threefolds, Brunner-Roggenkamp exhibit that the same considerations apply in all respects to higher-dimensional Calabi-Yaus, and so we find it equally interesting to study the veracity of this conjecture in higher dimensions. Note that in our phrasing above, we assume that $\gamma^{d, 0} \ne 0$ to avoid the possibility of non-isolated attractor points. 

We briefly give some more context to the above conjecture for the interested reader; this discussion will play no role in the proof and may be skipped without consequence. 

First, we note that the failure of the Attractor Conjecture should quite reasonably have been expected. Let us return for a moment to the case of Calabi-Yau threefolds, where again the attractor condition is that the two-dimensional vector space $H^{3,0}(X) \oplus H^{0,3}(X)$ contains a vector in the integral lattice. If instead we impose the stronger condition that it contains a rank-two sublattice, known as the \emph{rank-two attractor} condition, then the standard conjectures tell us that we expect $H^{3,0} \oplus H^{0,3}$ to split off as a motive, a condition which moreover would take place at algebraically-defined points. So, rank-two attractor points should certainly be algebraic; by contrast, the  attractor condition itself is too weak to suggest any motivic splitting, and one should expect no particular algebraicity in general. 

Second, as remarked above, Moore also conjectures in~\cite{moore} not just algebraicity of the attractor points but also algebraicity of their periods. In fact, the paragraph above already suggests distrust of this conjecture as well: when we have a motive that splits off, for example an $H^{1,2}(X) \oplus H^{2,1}(X)$ motive of a (Tate-twisted) CM elliptic curve in the case when $b_3(X) = 4$, there is no reason to suspect that the period of said CM elliptic curve need always be algebraic. And indeed, in examples with such motivic splitting such as Klemm-Scheidegger-Zagier's study of the conifold point of the mirror quintic or the examples of~\cite{CdlOEvS}, said periods have (numerically) been found to agree with the expected special values of the appropriate $L$-function from the motivic splitting. Hence, morally, our main theorem should be thought of as analogous (or even \emph{mirror symmetric}) to proving transcendence of special values of $L$-functions.

\section{Dolgachev Calabi-Yau varieties}
\subsection{Defining the Calabi-Yaus}\label{subsection:dolgcys}
We will consider a family of Calabi-Yau varieties constructed as crepant resolutions of $n$-fold cyclic covers of projective space. Most of the discussion holds for any $n \ge 2$, although the case $n = 2$ is completely classical and returns the Legendre family of elliptic curves; we hence restrict to $n \ge 3$ simply for convenience. At a crucial point, we find that the cases $n = 3, 4, 6$ are distinguished, giving rise to ``arithmetic'' families of Calabi-Yau varieties (e.g., these are exactly the cases for which the resulting $\mc{M}_{\mathrm{CY}}$ is Shimura); all three of these hence display qualitatively different behavior and we note at the appropriate point where the condition $n \ne 3, 4, 6$ is crucial to have the ``nonarithmeticity'' phenomenon of the statement of our main Theorem~\ref{thm:main}.

As promised, we consider the family of Calabi-Yau varieties given by $n$-fold cyclic covers of $\mb{P}^{2n-3}$ branched along $2n$ hyperplanes; this construction is due to Dolgachev, and we first collect here some basic properties about the Hodge structure of such a Calabi-Yau. So, consider $2n$ points $x_1, \cdots, x_{2n}$ in $\mb{P}^1$. Then, we may certainly consider the curve $C$ given by the $n$-fold cyclic cover of $\mb{P}^1$ branched at those $2n$ points; more precisely, we mean the cover determined by the same $n$-cycle monodromy about each branch point in the base, or simply the smooth, projective curve $C$ whose affine model is given by 
\begin{equation}\label{eqn:affinemodel} 
C^{\circ} = \{y^n = \prod_{i = 1}^{2n} (x - x_i)\}.
\end{equation}
Then, by construction, $H^1(C, \mb{Q})$ has both a Hodge splitting and a $\mb{Z}/n$-action; in fact, these two structures are compatible. More precisely, if we let $\zeta = e^{2\pi i/n}$, we know that  $H^1(C, \mb{Q}(\zeta))$ splits into eigenspaces for the $\mb{Z}/n$-action. Let $\mu \in \mb{Z}/n$ be the generator which acts on $C$ by 
\[
y \mapsto \zeta y.
\]
 
\begin{defn}
Let $$H^1(C)[i] \subset H^1(C; \mb{Q}(\zeta))$$ be the sub $\mb{Q}(\zeta)$-vector space  such that $\mu$ acts by $\zeta^i$.
\end{defn}
 Then $H^1(C)[i]$ is a $\mb{Q}(\zeta)$-sub-Hodge structure of $H^1(C, \mb{Q}(\zeta))$ in the sense that we have the decomposition
$$H^1(C)[i] \underset{\mb{Q}(\zeta)}{\otimes} \mb{C} \simeq H^{1, 0}(C)[i] \oplus H^{0, 1}(C)[i]$$ 
given by  the Hodge splitting of $H^1(C; \mb{C})$.

\begin{prop}\label{prop:hodgenumberscurve}
The Hodge numbers of $H^1(C)[i]$ are given by $(2i - 1, 2(n-i)-1)$, for $1 \le i \le n - 1$; that is 
\[
\dim_{\mb{C}}H^{1,0}(C)[i]=2i-1, \ \dim_{\mb{C}}H^{0,1}(C)[i]=2(n-i)-1.
\]
\end{prop}
We refer the reader to  \cite[Lemma 4.2]{looijenga} for this computation. Note that the  $i = 0$ eigenspace is trivial as  any   cohomology class in this space comes  from $H^1(\mb{P}^1) \simeq 0$ via pullback. 

\begin{constr} 
It remains to introduce the Calabi-Yau $(2n-3)$-folds $X$. Here we follow the treatment in \cite[Sections 2, 3]{sxzminimal}.  For a collection of $2n$ hyperplanes $(H_1, \cdots, H_{2n})$ of $\mb{P}^{2n-3}$, we say that they are in general position if no $2n-2$ of them intersect at any point; that is, there are no unexpected intersections between the $H_i$'s. Then we may define an $n$-fold cyclic cover $X'$ of $\mb{P}^{2n-3}$ branched along these hyperplanes. 

We give the rigorous construction here. For a line bundle $L$ on an arbitrary variety $Y$ and a positive integer $n$, consider the rank $n$ vector bundle 
\[
\mc{E}\defeq \mc{O}\oplus L^{\vee}\oplus \cdots \oplus (L^{\vee})^{\otimes n-1}
\]
on $Y$; here $L^{\vee}$ denotes the dual of $L$. Now given a section of $L^{\otimes n}$, or equivalently a  map
\[
(L^{\vee})^{\otimes n}\rightarrow\mc{O},
\]
we may define an algebra structure on $\mc{E}$ in the obvious way, and therefore  we may form the variety 
\[
X\defeq \Spec(\mc{E});\,
\]
and by construction $X$ admits a map to $Y$; in fact, this is a cyclic $n$-fold covering. In other words, a section $\sigma \in \Gamma(L^{\otimes n})$ defines a cyclic $n$-fold cover $X\rightarrow Y$.
\end{constr}
\begin{defn}\label{defn:dolgachevcy}
For a collection of $2n$ points $p_1, \cdots p_{2n}$ on $\mb{P}^1$, we may consider $2n$ hyperplanes on $\mb{P}^{2n-3}$ as follows: recall that there is an isomorphism
\begin{equation}\label{eqn:symmetricpower}
\Sym^{2n-3}\mb{P}^1\cong \mb{P}^{2n-3},
\end{equation}
and also that, for each $i=1, \cdots, 2n$, the set of points of the form $\{p_i\}\times \mb{P}^1\times\cdots \times \mb{P}^1$ on the left hand side of (\ref{eqn:symmetricpower}) gives a hyperplane on the right hand side. Therefore we obtain $2n $ hyperplanes $H_1, \cdots, H_{2n}$ on $\mb{P}^{2n-3}$, in general position.

 We  now apply the above construction  to $Y=\mb{P}^{2n-3}$, $L=\mc{O}(2)$, and $\sigma \in \Gamma(\mc{O}(2n))$ such that the zero locus of $\sigma$ is precisely 
\[
D\defeq \sum_{i=1}^{2n} H_i \subset \mb{P}^{2n-3}
\]
to obtain a cyclic $n$-fold covering of $\mb{P}^{2n-3}$, which we denote by $X'$.
\end{defn}
Note that a pleasant  computation shows that the canonical bundle of the cyclic cover $X'$ defined above is trivial: indeed, using the formula $K_{X}=\pi^*K_Y+R$ for a covering map $\pi:X\rightarrow Y$, where $R\subset X$ is the ramification divisor, we deduce that  
\begin{align*}
K_{X'}&=\pi^*((2-2n)H)+(2n)(n-1)H'\\
   &=(n(2-2n)+2n(n-1))H'\\
   &=0,
\end{align*}
and so the canonical class of $X'$ is trivial; here $H$ denotes a hyperplane class in $\mb{P}^{2n-3}$ and $H'$ is the class of one of the $n$ components of the pullback of $H$.

Denote the moduli space of collections $(H_1, \cdots, H_{2n})$ of hyperplanes of $\mb{P}^{2n-3}$ by $\mc{H}_n$. The following was proved by Sheng-Xu-Zuo \cite[Corollary 2.6]{sxzminimal}:

\begin{thm}\label{thm:crepant} 
Denote by $f': \mc{X}'\rightarrow \mc{H}_n$ the   family of $2n-3$-folds over the moduli of hyperplane arrangements constructed above. Then 
\begin{enumerate}
    \item there is a  family of smooth Calabi-Yau $(2n-3)$-folds 
    \[
    f:\mc{X}\rightarrow \mc{H}_n,
    \]
    as well as a commutative diagram
    \begin{equation}
    \begin{tikzcd}[column sep=small]
\mc{X} \arrow[rr, "\sigma"] \arrow[dr, "f"'] & & \mc{X}' \arrow[dl, "f'"]\\
& \mc{H}_n &
\end{tikzcd}
\end{equation}
where $\sigma$ is a simulatneous crepant resolution;

\item the middle degree Hodge structures of the families $\mc{X}'$ and $\mc{X}$ agree:
\[
R^{2n-3}f'_*\mb{Q}\cong R^{2n-3}f_*\mb{Q};
\]
\item 
furthermore, the family  $f$ is maximal in the sense that the Kodaira-Spencer map is an isomorphism at each point $p\in \mc{H}_n$.
\end{enumerate}
\end{thm}
\begin{defn}
We refer to the varieties constructed in Theorem \ref{thm:crepant} as the Dolgachev Calabi-Yau varieties. For $X'=f'^{-1}(p)$  the $(2n-3)$-fold parametrized by some $p\in \mc{H}_n$, we denote by $X\defeq f^{-1}(p)$ the corresponding  crepant resolution. 
\end{defn}
\begin{rmk}
This is a slight abuse of terminology since the crepant resolution  $\sigma : \mc{X}\rightarrow \mc{X}$ is not unique; on the other hand any two resolutions are of course birational to each other.
\end{rmk}

\subsection{Hodge structures of Dolgachev Calabi-Yaus}
Now we relate the curves constructed earlier as branched covers of $\mb{P}^1$ to the Calabi-Yau varieties constructed as covers of $\mb{P}^{2n-3}$. 
Recall from Theorem \ref{thm:crepant} that the cyclic cover $X'$ has a  crepant resolution $X$, which is Calabi-Yau and  we have an isomorphism
\[
H^{2n-3}(X; \mb{Q}) \simeq H^{2n-3}(X'; \mb{Q}).
\]
The right hand side  has a natural $\mb{Z}/n$-action since it arises as a cyclic cover, and therefore the left hand side does as well.  We may therefore  decompose $H^{2n-3}(X)$ into eigenspaces as follows. As in Section \ref{subsection:dolgcys} we fix a generator $\mu\in \mb{Z}/n$.
\begin{defn} 
We define  
$$H^{2n-3}(X)[i] \subset H^{2n-3}(X; \mb{Q}(\zeta))$$ 
as the sub-vector space over $\mb{Q}(\zeta)$ on which $\mu$ acts by $\zeta^i$.
\end{defn}
Then we have the crucial relationship between the Hodge structures of $C$ and $X$: 
\begin{lem}[{\cite[Proposition 2.2, Remark 2.2]{sxzminimal}}]\label{lemma:hodgestructurecy}
We have the following isomorphism of Hodge structures:
$$H^{2n-3}(X)[i] \simeq \bigwedge \! ^{2n-3} H^1(C)[i].$$ 
\end{lem}
\begin{rmk}
Perhaps a more intrinsic way of phrasing the above lemma is that there is an isomorphism of $\mb{Q}$-Hodge structures with $\mc{A}\defeq \mb{Q}(X)/(X^n-1)$-action
\[
H^{2n-3}(X, \mb{Q})\cong \bigwedge_{\mc{A}}\! ^{2n-3}H^1(C, \mb{Q}).
\]
In other words, we view $H^1(C, \mb{Q})$ as an $\mc{A}$-module, where the generator $X\in \mc{A}$ acts through $\mu\in \mb{Z}/n$, and then we take its $(2n-3)\textsuperscript{rd}$ wedge power over $\mc{A}$.
\end{rmk}
We have the following simple 
\begin{cor}\label{cor:hodgestructurecy}
The Hodge structure $H^{2n-3}(X)\otimes \mb{Q}(\zeta)$ decomposes as a sum 
\[
H^{2n-3}(X)\otimes \mb{Q}(\zeta)\cong \bigoplus_{i=1}^{n-1} V_i
\]
where $V_i$ is a $\mb{Q}(\zeta)$-Hodge structure,  concentrated only in Hodge degrees 
\[
(p,q)=(2i-2, 2n-2i-1)\ \text{and}\  (2i-1, 2n-2i-2);
\]
furthermore the dimensions of these pieces of the Hodge decomposition are
\[
2i-1 \ \text{and} \  2n-2i-1,
\]
respectively.
\end{cor}
\begin{proof}
Indeed, we define the Hodge structures $V_i$ to be  $\bigwedge^{2n-3}H^1(C)[i]$ from Lemma \ref{lemma:hodgestructurecy}. By  Proposition \ref{prop:hodgenumberscurve} we may write the Hodge decomposition of $H^1(C)[i]$ as 
\[
H^1(C)[i]\otimes \mb{C}\cong H^{1,0}\oplus H^{0,1},
\]
where we have omitted the dependence on $i$ on the right hand side, and 
\[
\dim H^{1,0}=2i-1, \ \dim H^{0,1} = 2(n-i)-1.
\]

For convenience let us pick a  basis $\{e_i\}$ (respectively  $\{f_j\}$) for $H^{1,0}$ (respectively $H^{0,1}$). Since the dimension of $H^1(C)[i]$ is $2n-2$, upon taking the $(2n-3)$-th wedge power, the only non-zero elements obtained by wedging together $e_i$'s and $f_j$'s must  omit precisely one $e_i$ or one $f_j$. Therefore the Hodge degrees of such an element are either 
\[
(p,q)=(2i-2, 2(n-1)-1) \ \text{or} \ (2i-1, 2(n-i)-2);
\]
furthermore there are $2i-1$ (respectively $2(n-i)-1$) choices of an $e_i$ (respectively $f_j$) to omit, and therefore the Hodge numbers are $2i-1$ (respectively $2n-2i-1$), as claimed.
\end{proof}
\begin{notation}\label{notation:hodge}
We will sometimes use the following piece of notation for bookkeeping when dealing with these Hodge numbers. We record the dimensions in a $(n-1)\times 2$ matrix 
\[
\begin{pmatrix}
\dim H^{1,0}(C)[1] & \dim H^{0,1}(C)[1]\\
\dim H^{1,0}(C)[2] & \dim H^{0,1}(C)[2]\\
\vdots & \vdots \\
\dim H^{1,0}(C)[n-1] & \dim H^{0,1}(C)[n-1]
\end{pmatrix}.
\]
For example, in the case $n=5$, it follows from Proposition  \ref{prop:hodgenumberscurve} that the above matrix is 
\[
\begin{pmatrix}
1 & 7 \\
3 & 5\\
5 & 3\\ 
7 & 1
\end{pmatrix}.
\]
\end{notation}
\begin{rmk} 
We therefore  verify from the case of $i = n-1$ that $h^{2n-3, 0} = 1$, as expected for a Calabi-Yau variety of dimension $2n-3$; moreover, we find $\dim \mc{M}_{\mathrm{CY}} = h^{2n-4, 1} = 2n-3$, which is notably the same as the dimension of the $\mc{M}_{0, 2n}$, the moduli of $2n$ points in $\mb{P}^1$. Indeed, this construction exactly accounts for the full moduli space of Calabi-Yau varieties so constructed, i.e. $$\mc{M}_{0, 2n} \overset{\sim}{\to} \mc{M}_{\mathrm{CY}}.$$ 
\end{rmk}
As a consistency check let us  see that the dimensions of  $\mc{M}_{0,2n}$ and $\mc{H}_n$ agree: indeed, each moduli space parametrizes  hyperplanes inside a projective space modulo the action of a projective linear group, and the coincidence of the dimensions is the equality
\[
(2n)-(3)=(2n-3)(2n)-((2n-2)^2-1);
\]
here the first term on each side of the equation is the number of moduli for the hyperplanes, while  the second term is the dimension of the projective linear group.
\begin{rmk}
Note that in the case when  $n$ is not a prime, for any Dolgachev Calabi-Yau variety $X$ there exists classes $\gamma\in H^{2n-3}(X, \mb{Q})$ with no component in $H^{2n-3}(X)[1]$: indeed, just take any element in $H^{2n-3}(X, \mb{Q}(\zeta))[i]$ for some $i$ not coprime to $n$, and take the sum of all of its Galois conjugates. In fact, when we take the parallel transport of such a  class $\gamma$, it will continue to have no component in  $H^{2n-3}(X)[1]$ for any $X$; this shows that the condition $\gamma^{2n-3,0}\neq 0$ condition is necessary in Conjecture \ref{conj:mooreattractor}.
\end{rmk}
\subsection{The attractor condition for Dolgachev Calabi-Yau varieties}
We are now finally in the position to study the attractor condition for Calabi-Yau varieties $X$ constructed as above. We first show that the attractor condition is equivalent to a condition on the periods of the associated curve $C$ as follows:

\begin{lem}\label{attractorcondition} The variety $X$ satisfies the attractor condition if and only if there exists $\omega \in H^{1, 0}(C)[1] \cap H^1(C)[1]$. \end{lem}

Indeed, recall that $H^{1, 0}(C)[1]$ is one-dimensional, so the above condition is that the above subspace of $H^1(C)[1]$, {\it a priori} only defined once one tensors up to $\mb{C}$, is in fact defined over $\mb{Q}(\zeta)$. 

\begin{proof} Suppose $X$ satisfies the attractor condition. Then there exists $\gamma \in H^{2n-3}(X; \mb{Z})$ orthogonal to $H^{2n-4, 1}(X)$; equivalently, $\gamma$ is orthogonal to $H^{1, 2n-4}(X)$. Recall from the discussion above that $H^{1, 2n-4}(X)$ is contained within the $H^{2n-3}(X)[1]$ eigenspace. The distinct $\mu$-eigenspaces are certainly orthogonal under the intersection pairing, and if $\gamma_i \in H^{2n-3}(X)[i]$ denotes the summand of $\gamma$ under the decomposition of $H^{2n-3}(X; \mb{Q}(\zeta))$, the above condition is equivalent to $\gamma_1$ orthogonal to $H^{1, 2n-4}(X)$. As the Hermitian pairing is perfect on $H^{1, 2n-4}(X)$, $\gamma$ cannot have any support within said Hodge summand of $$H^{2n-3}(X)[1]_{\mb{C}} \simeq H^{1, 2n-4}(X) \oplus H^{0, 2n-3}(X),$$ and so we must have $\gamma_1 \in H^{0, 2n-3}(X)$. But $\gamma_1$ was defined as an element of the vector space $H^{2n-3}(X)[1]$, a vector space defined over $\mb{Q}(\zeta)$, and so both $H^{0, 2n-3}(X)$ and $H^{1, 2n-4}(X)$, as the orthogonal complement of $H^{0, 2n-3}(X)$ under the intersection pairing restricted to $H^{2n-3}(X)[1]$, are defined over $\mb{Q}(\zeta)$. But then \begin{eqnarray*} H^{1, 2n-4}(X) &\simeq& \wedge\!^{2n-2} H^{0, 1}(C)[1] \otimes H^{1, 0}(C)[1] \\ &\simeq& \Big(H^{0, 1}(C)[1]\Big)^{\vee} \otimes \Big( \det H^{0, 1}(C)[1] \Big) \otimes H^{1, 0}(C)[1] \\ &\simeq& \Big(H^{0, 1}(C)[1] \Big)^{\vee} \otimes \det H^1(C)[1] \end{eqnarray*} as a subspace of $$H^{2n-3}(X)[1] \simeq \wedge\!^{2n-3} H^1(C)[1] \simeq \Big(H^1(C)[1]\Big)^{\vee} \otimes \det H^1(C)[1].$$ Above, we use the notation $\det V = \wedge^{\dim V} V$ and the isomorphism $$\wedge\!^{\dim V - 1} V \simeq V^{\vee} \otimes \det V.$$ In any case, we have that the decomposition $$H^{2n-3}(X)[1]_{\mb{C}} \simeq H^{1, 2n-4}(X) \oplus H^{0, 2n-3}(X)$$ is isomorphic to the decomposition $$\Big( \Big(H^1(C)[1]\Big)^{\vee} \otimes \det H^1(C)[1] \Big)_{\mb{C}} \simeq \Big(  \Big(H^{0, 1}(C)[1] \Big)^{\vee} \otimes \det H^1(C)[1] \Big) \oplus \Big(  \Big(H^{1, 0}(C)[1] \Big)^{\vee} \otimes \det H^1(C)[1] \Big)$$ induced from the Hodge splitting of $H^1(C)[1]_{\mb{C}}$. As $H^1(C)[1]$ and hence $\det H^1(C)[1]$ are defined over $\mb{Q}(\zeta)$, however, the condition that the first decomposition be defined over $\mb{Q}(\zeta)$ is equivalent to the condition that the second decomposition be defined over $\mb{Q}(\zeta)$, which in particular implies that there exists some $\omega \in H^{1, 0}(C)[1] \cap H^1(C)[1]$.

Conversely, given such an $\omega$, we have that the subspace $H^{1, 0}(C)[1] \subset H^1(C)[1]_{\mb{C}}$ is in fact defined over $\mb{Q}(\zeta)$ and hence so is $H^{0, 1}(C)[1]$ as its orthogonal complement; as above, the decomposition $H^{2n-3}(X)[1]_{\mb{C}} \simeq H^{1, 2n-4}(X) \oplus H^{0, 2n-3}(X)$ is then also defined over $\mb{Q}(\zeta)$. Then take some $\gamma_1 \in H^{0, 2n-3}(X)$ defined over $\mb{Q}(\zeta)$ so that by construction, $\gamma_1$ is orthogonal to $H^{1, 2n-4}(X)$, and consider the Galois conjugates $\gamma_i$ under the action of $\Gal(\mb{Q}(\zeta)/\mb{Q})$ on $H^{2n-3}(X; \mb{Q}(\zeta))$. These Galois conjugates will lie within the $H^{2n-3}(X)[i]$ eigenspaces for values of $i$ coprime to $n$ and thereby be concentrated in Hodge summands away from the $(1, 2n-4)$ and $(0, 2n-3)$ summands, so that if we now define $\gamma = \sum_i \gamma_i$, the Galois-theoretic construction will give us $\gamma \in H^{2n-3}(X; \mb{Q})$ while its summand in the $H^{2n-3}(X)[1]$ eigenspace is still the original $\gamma_1$ we started with. As such, scaling $\gamma$ as necessary so it in fact lies in $H^{2n-3}(X, \mb{Z})$, we have produced some integral cohomology class orthogonal to $H^{1, 2n-4}(X)$, or equivalently $H^{2n-4, 1}(X)$, as desired. \end{proof}

\subsection{Algebraicity of the associated curve}
In this section we show that the algebraicity of the Dolagchev Calabi-Yau variety implies that of the curve associated to it. Therefore  to show that Dolgachev Calabi-Yau varieties provide counterexamples to the Attractor Conjecture it suffices to show that for the attractor varieties, the assocaited curves are not defined over $\overline{\mb{Q}}$.

\begin{prop} $X$ is defined over $\overline{\mb{Q}}$ if and only if $C$ is defined over $\overline{\mb{Q}}$. \end{prop}

\begin{proof} We first do the easier direction, so suppose $C$ is defined over $\overline{\mb{Q}}$. Let $\mu: C \to C$ 
denote a generator of the $\mb{Z}/n$ action, which must also be defined over $\overline{\mb{Q}}$: we spell out this argument here as we will use its basic idea (``spreading out'') frequently. So, consider $\mu$ as a point of the quasiprojective $\overline{\mb{Q}}$-scheme $\mathrm{Aut}(C)$. If the field of definition $K$ of $\mu$ is larger than $\overline{\mb{Q}}$, and in particular contains some pure transcendental extension thereof, we may freely specialize that transcendental variable to produce a family of automorphisms of $C$, but any curve has only finitely many automorphisms. Hence $\mu$ must have been defined over $\overline{\mb{Q}}$. But now the morphism $C \to C / \mu \simeq \mb{P}^1$ is defined over $\overline{\mb{Q}}$, and so the $2n$ points $x_1, \cdots, x_{2n} \in \mb{P}^1$ of ramification are defined over $\overline{\mb{Q}}$ (after an appropriate automorphism of $\mb{P}^1$). But then it is clear that the cover $X'$ and all the blow-up centers within $X'$ are defined over $\overline{\mb{Q}}$, and hence so is $X$. 

The more interesting direction is the reverse argument, where we begin by supposing that $X$ is defined over $\overline{\mb{Q}}$. The morphism $X \to \mb{P}^{2n-3}$ corresponds to some line bundle $\mc{L} \in \mathrm{Pic}\,X$ given by the pullback of $\mc{O}(1)$, but note that $\mathrm{Pic}\,X \simeq H^2(X; \mb{Z})$ is simply a discrete set of points as a scheme over $\overline{\mb{Q}}$, and hence all points must be defined over $\overline{\mb{Q}}$. 
\begin{claim}
The complete linear system of $\mc{L}$ defines precisely the morphism $X\rightarrow \mb{P}^{2n-3}$.
\end{claim}
\begin{proof}
It suffices to show that the pullback map induces an isomorphism 
\[
\Gamma(X, \mc{L})\cong \Gamma(\mb{P}^{2n-3}, \mc{O}(1)).
\]
Recall that  we have the factorization 
\[
X\xrightarrow{\sigma} X'\xrightarrow{\alpha} \mb{P}^{2n-3},
\]
where $\sigma$ is the crepant resolution from Theorem \ref{thm:crepant}, and 
\[
\alpha: X'\rightarrow \mb{P}^{2n-3}
\]
denotes  the $n$-fold covering  of $\mb{P}^{2n-3}$  from Definition \ref{defn:dolgachevcy}. We first show that 
\[
\Gamma(X', \mc{L}')\cong \Gamma(\mb{P}^{2n-3}, \mc{O}(1)),
\]
where $\mc{L}'\defeq \alpha^*\mc{O}(1)$.  By construction of $X'$ (see Definition \ref{defn:dolgachevcy} and the paragraph preceding it), 
\[
\alpha_*\mc{O}_{X'}\cong \mc{O}_{\mb{P}^{2n-3}}\oplus L^{\vee}\oplus \cdots \oplus (L^{\vee})^{\otimes n-1},
\]
where $L^{\vee}\cong \mc{O}(-2)$. Therefore
\[
\alpha_*\mc{L}\cong \mc{O}_{\mb{P}^{2n-3}}(1)\oplus \mc{O}_{\mb{P}^{2n-3}}(-1)\oplus \cdots \oplus \mc{O}_{\mb{P}^{2n-3}}(-2n+3),
\]
and hence $\Gamma(X', \mc{L}')=\Gamma(\mb{P}^{2n-3}, \mc{O}(1))$ as required. On the other hand, $X$ is obtained from $X'$ by blowing up along subvarieties of codimension at least two, and we claim 
\[
\Gamma(X, \mc{L})\cong \Gamma(X', \mc{L}')
\]
as well. Indeed, as $X$ and $X'$ fail to be isomorphic only in codimension two, this statement would follow from (algebraic) Hartogs' Lemma provided $X$ and $X'$ are both normal. That $X$ is normal follows from its smoothness, while $X'$ is normal as it is both $R_1$ and $S_2$. Indeed, its singular set has codimension two, while it is $S_2$ given its construction as a hypersurface in a smooth ambient variety (the total space of a line bundle over $\mb{P}^{2n-3}$). 
\end{proof}
Hence $\mc{L}$, and its linear system $X \to \mb{P}^{2n-3}$ is defined over $\overline{\mb{Q}}$, and so we learn that the ramification locus with irreducible components the $2n$ hyperplanes $H_1, \cdots, H_{2n}$ may be taken to be defined over $\overline{\mb{Q}}$ -- i.e. are defined over $\overline{\mb{Q}}$ after, possibly, an application of some $PGL_{2n-2}$ projective transformation to their original definition as corresponding to the points $x_i$. However, this condition is precisely the same as that the original $2n$ points may be taken to be defined over $\overline{\mb{Q}}$, i.e. possibly after some $PGL_2$ transform, or equivalently that all their cross-ratios are in $\overline{\mb{Q}}$, and so it is then easy to reconstruct $C$ over $\overline{\mb{Q}}$. Indeed, the map from the $2n$ points to the $2n$ hyperplanes (or $2n$ points in a dual projective space) may be regarded as a morphism between (open loci of) $\mathrm{Sym}^{2n-3} \mb{P}^1 / S_3 \to \mb{P}^{2n-3} / S_{2n-2}$ (by using the simple $3$- or $2n-1$-transitivity of the $PGL_2$- and $PGL_{2n-2}$-actions, respectively) which is explicitly defined over $\overline{\mb{Q}}$. Indeed, one may write down this map in explicit coordinates: we refer the reader to \cite[Claim 3.6]{sxzminimal}.  






\end{proof}

\subsection{Attractors are dense}\label{subsection:dense}
In this section  we show  that the attractor points are Zariski dense in moduli space. This fact  will be used when we apply the Zilber-Pink conjecture.

We define the auxiliary space
\[
\mathcal{M}_{0,2n}^{'} \defeq \{ (s, \omega)| s\in \mc{M}_{0,2n},\ \omega \in H^{1,0}(C_s)[1],\ \omega \neq 0\}
\]
where we have denoted by $C_s$ the $n$-fold cover of $\mb{P}^1$ branched at the configuration of $2n$ points given by $s \in \mc{M}_{0,2n}$. Recall that $H^{1,0}(C_s)[1]$ is a one dimensional and hence  $\mathcal{M}_{0,2n}^{'}$ is a $\mb{G}_m$-bundle over moduli space. Also let $\widetilde{\mathcal{M}}_{0,2n}^{'}$ denote the universal cover of $\mathcal{M}_{0,2n}^{'}$; on this universal cover we have a well defined basis of the cohomology group $H^1(C, \mb{Q}(\zeta_n))[-1]$ which we denote by $\gamma_1, \cdots , \gamma_{2n-2}$.

We may now consider the so-called Schwarz map defined as follows:
\begin{align*}
\pi: \widetilde{\mathcal{M}}_{0,2n} &\rightarrow \mb{C}^{2n-2}\\
(s, \omega) & \mapsto   \Big(\int_{\gamma_1}\omega , \cdots ,\int_{\gamma_{2n-2}}\omega \Big).
\end{align*}

Now by Lemma \ref{attractorcondition} we have that a point $(s,\omega) \in \widetilde{\mathcal{M}}_{0,2n}$ is an attractor point (more precisely the point $s$ gives rise to an attractor CY and $\omega$ witnesses this) if and only if $\pi((s,\omega))$ has coordinates in $\mb{Q}(\zeta) \subset \mb{C}$. Now note that $\pi$ is a holomorphic  local homeomorphism, and so the image $\pi(\widetilde{\mathcal{M}}_{0,2n})$ contains some open ball inside $\mb{C}^{2n-2}$. Since $\mb{Q}(\zeta_n) \subset \mb{C}$ is dense, we have that the attractors are topologically dense, and hence Zariski dense as well. To summarize we have the following:

\begin{prop}
The attractor points are Zariski dense in the moduli space $\mc{M}_{\mathrm{CY}}.$
\end{prop}

\section{Reduction to Shimura theory}
\subsection{Algebraic attractors split off CM abelian varieties}
In this subsection we show that if an attractor is algebraic, then the Jacobian of the corresponding curve $C$ must split off CM factors.

We make use of the following theorem of Shiga-Wolfart [ref]\cite[Proposition 3]{shigawolfart}, a consequence of the analytic subgroup theorem of W\"ustholz [ref]:

\begin{thm}[Shiga-Wolfart]\label{shigawolfart}  Suppose $A$ is an abelian variety over $\overline{\mb{Q}}$ endowed with $\omega \in \Gamma(A, \Omega^1_A)_{\overline{\mb{Q}}}$ such that for any two classes $\gamma_1, \gamma_2 \in H^1(A; \mb{Z})$, we have that the period ratios are algebraic:
\[
\frac{\langle \omega, \gamma_1 \rangle}{\langle \omega, \gamma_2 \rangle} \in \overline{\mb{Q}}.
\]
Then $A$ has complex multiplication. Moreover, if $K$ is the number field generated by the period ratios above then the CM field of $A$ is precisely $K$. \end{thm}

The last sentence above is not part of the proposition cited but follows from their  proof, which uses the analytic subgroup theorem to directly construct endomorphisms of $A$ from the hypothesized period relations. 
More generally, it follows that if $A$ splits in the isogeny category as some product of abelian varieties $A_i$, then for all $A_i$ on which the $\omega$ above is supported (i.e. restricts to nontrivially), $A_i$ has complex multiplication by some (possibly varying with $i$) subfield of $K$. Applied to the case above, we have the immediate consequence:
\begin{prop}\label{man}If $X$ as above satisfies the attractor condition, then $\Jac C$ has a summand $A$ in the isogeny category such that the following  conditions hold:
\begin{enumerate}
    \item $\omega$ restricts non-trivially to $A$;
    \item $A$ has complex multiplication (in the isogeny category) by $\mb{Q}(\zeta).$
\end{enumerate}
\end{prop}
\begin{proof}
Let $A_1$ denote the simple abelian variety, which is a summand of $A$ in the isogeny category, on which $\omega$ restricts non-trivially. Then the Hodge structure of $A_1$ is a $\mb{Q}$-sub-Hodge structure of $H^1(A, \mb{Q})$, and hence must contain all the Galois conjuagtes of $\omega$. Since $\omega$ lives in $H^1(A)[1]$ on which $\mu$ acts by a primitive root of unity,  we have 
\[
\dim A_1\geq \phi(n)/2.
\]
On the other hand, since $X$ satisfies the attractor condition, all the periods 
\[
\langle \omega , \gamma\rangle \ \text{for} \ \gamma\in H^1(A_1, \mb{Q})
\]
are contained in $\mb{Q}(\zeta)$, and applying  Theorem \ref{shigawolfart} to $A_1$ and $\omega$ we deduce that $A_1$ has complexmultiplication by a subfield of $K=\mb{Q}(\zeta)$, and by the inequality on the dimension above we conclude that $A_1$ has complex multiplication by $K$, as required.  
\end{proof}

In fact, it is possible to be more precise still in this case: the endomorphisms constructed from the W\"{u}stholz analytic subgroup theorem commute with the $\mb{Z}/n$ cyclic action 
and so each CM abelian variety $A_i$ produced from the Shiga-Wolfart argument continues to respect the $\mb{Z}/n$-equivariant structure. But $\omega$ is the unique holomorphic form in its eigenspace, up to scaling, and so there can only be one $A_i$ upon which $\omega$ is supported. 

The theory of complex multiplication now tells us that there is some finite list of abelian varieties $A$, up to isogeny, with CM  related by  $\mb{Q}(\zeta)$ as above. The dimension of $A$ is $\frac{1}{2} \dim [\mb{Q}(\zeta):\mb{Q}] = \frac{1}{2} \phi(n)$, and so we find that the following characterization of the attractor points in $\mc{M}_{\mathrm{CY}}$:

\begin{thm} \label{thm:splittingcm}
The attractor points for the Dolgachev family of CYs considered here are exactly the intersection of $\mc{M}_{\mathrm{CY}} \simeq \mc{M}_{0, 2n}$ with the Hecke translates of the sub-Shimura varieties above of $\mc{A}_{(n-1)^2}$ under $\mc{M}_{0, 2n} \to \mc{M}_{(n-1)^2} \to \mc{A}_{(n-1)^2}.$ \end{thm}

\subsection{Prym varieties}\label{section:prym}
In fact, while $\mc{M}_{0, 2n}$ does naturally map to the Shimura variety parametrizing $(n-1)^2$-dimensional (principally polarized) abelian varieties as above, for the application of the Zilber-Pink conjecture it is necessary to refine this map slightly, especially when $n$ is not a prime. The end result will be a map to a PEL type Shimura variety instead of simply $\mc{A}_{(n-1)^2}$.
Therefore in  this section we study the construction of Prym varieties, which is a certain quotient of the Jacobian. 

Recall that, for $n\geq 2$,  $C\rightarrow \mb{P}^1$ denotes the cyclic $n$-fold covering of $\mb{P}^1$ branched at $2n$ points, whose affine model is given in (\ref{eqn:affinemodel}); moreover there is  an action of $\mb{Z}/n$ on $C$. Now suppose we have a divisor $n'$ of $n$, and let $\mb{Z}/n'\subset \mb{Z}/n$ denote the unique order $n'$ subgroup of $\mb{Z}/n$; we fix a generator $\mu\in \mb{Z}/n$ as before, and further denote by $\mu'\defeq (n/n')\mu$, which is  a generator  of this $\mb{Z}/n'$ subgroup.
\begin{defn}
For each $n'$ dividing $n$, define 
\[
C'\defeq C/(\mb{Z}/n'),
\]
where $\mb{Z}/n'$ acts on $C$ via the inclusion $\mb{Z}/n'\subset \mb{Z}/n$. Also let 
\[
\pi_{n'}: C\rightarrow C'
\]
denote the natural quotient map. By further quotienting by a $\mb{Z}/(n/n')$, we also have a map $C'\rightarrow \mb{P}^1$, which is a cyclic $n/n'$-fold covering.
\end{defn}
\begin{prop}
Let $J_{n'}$ denote the cokernel of  the pullback map
\[
\pi^*_{n'}:\Jac(C')\rightarrow \Jac(C).
\]
Then its  cohomology  is given by 
\[
H^1(J_{n'}, \mb{Q}(\zeta))\cong \bigoplus_i H^1(C, \mb{Q}(\zeta))[in'];
\]
equivalently,  the above is the sum of all $H^1(C, \mb{Q}(\zeta))[j]$ where  $j$ satisfies 
\[
\zeta^{jn/n'}=1.
\]
\end{prop}
\begin{rmk}
As a consistency check, we see that the above sum is over $i=n', \cdots, (\frac{n}{n'}-1)n'$, and so there are $(n/n'-1)$ non-trivial summands, as expected, since \[
C'\rightarrow \mb{P}^1
\]
is now a $n/n'$-fold cyclic cover.
\end{rmk}
\begin{proof}
Applying the Riemann-Hurwitz formula to the covering map $C'\rightarrow \mb{P}^1$, we have 
\[
2-2g(C')=\frac{n}{n'}(2)-2n(\frac{n}{n'}-1),
\]
and hence 
\[
g(C')=(1+n)\big(1-\frac{n}{n'}\big).
\]
Here $g(C')$ denotes the genus of $C'$. On the other hand, recall that the points of the Jacobian of a curve are the degree zero divisors modulo rational equivalence, and therefore the image of the pullback map
\[
\pi_{n'}^*: \Jac(C')\rightarrow \Jac(C)
\]
is invariant under the $\mb{Z}/n'$-action. On the other hand $\pi_{n'}^*$ is injective, since if $D$ is a degree zero divisor on $C'$ such that 
\[
\pi^*_{n'}(D)=(f)
\]
for some rational function $f$ on $C$, then $f$ is invariant under the Galois group $\mb{Z}/n'$ of the covering map $C\rightarrow C'$, and therefore $f$ descends to $C$. Therefore the  map on homology induced by $\pi^*_{n'}$ is also injective, and lands inside the invariant subspace  $H_1(\Jac(C), \mb{Q}(\zeta))^{\mb{Z}/n'}$ (using $\mb{Q}(\zeta)$-coefficients). By the genus computation above, this gives an isomorphism
\[
H_1(\Jac(C'), \mb{Q}(\zeta))\cong H_1(\Jac(C), \mb{Q}(\zeta))^{\mb{Z}/n'}.
\]
Dualizing, we have $H^1(\Jac(C'), \mb{Q}(\zeta))$ being the coinvariants of the $\mb{Z}/n'$-action on $H^1(\Jac(C), \mb{Q}(\zeta))$, which gives the desired result: indeed, since the action of a generator $\mu'\defeq (n/n')\mu \in \mb{Z}/n'$ on $H^1(\Jac(C), \mb{Q}(\zeta)[j]$ is given by $\zeta^{nj/n'}$, the coinvariants are given by 
\[
\bigoplus_{i=1}^{n-1}H^1(\Jac(C), \mb{Q}(\zeta))[j]/(\zeta^{nj/n'}-1)H^1(\Jac(C), \mb{Q}(\zeta))[j],
\]
whose non-trivial summands are indexed by $i$ such that $\zeta^{jn/n'}=1$, as claimed.
\end{proof}
We immediately deduce the following simple
\begin{cor}
We denote by $\pi^*_{n'}$ the map on cohomologies induced by $\pi^*_{n'}$. Then the quotient of $H^1(C, \mb{Q}(\zeta))$ by the images of $\pi^*_{n'}$ for all proper divisors $n'$ (i.e. $n'\neq 1, n$) is precisely the sum 
\[
\bigoplus_{i\in (\mb{Z}/n)^{\times}} H^1(C, \mb{Q}(\zeta))[i].
\]

\end{cor}
The above can be refined integrally, or equivalently as a statement aboue abelian varieties.
\begin{defn}
We now   define the abelian variety
\[
\Prym\defeq \Jac(C)/\sum_{n'}\Imag(\pi^*_{n'}),
\]
and refer to it as the Prym variety. Here the sum is over proper divisors $n'$ as above.
\end{defn}
\begin{cor}\label{cor:prymhodge}
The abelian variety $\Prym$ has endomorphisms by $\mb{Q}(\zeta)$, and its $\mb{Q}(\zeta)$-Hodge structure is given by 
\[
\bigoplus_{i\in (\mb{Z}/n)^{\times}} H^1(C, \mb{Q}(\zeta))[i].
\] As such, it has dimension $(n-1)\phi(n)$.
\end{cor}

\subsection{PEL Shimura varieties}
Now that we have the neccesary statements on the Prym construction from Section \ref{section:prym}, we can define the refined period, whose image is a certain PEL type Shimura variety.
Denote by $V$ the $\mb{Q}$ subspace of $H^1(C, \mb{Q})$ such that 
\[
V\otimes \mb{Q}(\zeta)=\bigoplus_{r\in (\mb{Z}/n)^{\times}}V[r],
\]
where $V[r]\subset H^1(C, \mb{Q}(\zeta))$ denotes the $\zeta^r$ eigenspace for the action of $\mu\in \mb{Z}/n$; as before $V$ has an action  of $\mb{Q}(\zeta)$.

\begin{defn}
The abelian variety $\Prym$ from Section \ref{section:prym} furnishes us with an integral lattice 
\[
V_{\mb{Z}}\defeq H^1(\Prym, \mb{Z})\subset H^1(\Prym, \mb{Q})=V,
\]
equipped with a symplectic form $\Psi$. Let 
\[
\mb{S}\defeq \Res_{\mb{C}/\mb{R}}\mb{G}_m
\]
denote the Deligne torus, and let $\mf{H}$ denote the space of homomorphisms 
\[
h: \mb{S}\rightarrow \GSp(V_{\mb{R}}, \Psi)
\]
which define Hodge structures of type $(-1,0)+(0,-1)$ on $V_{\mb{Z}}$. The space $\mf{H}$ is isomorphic to the Siegel upper half space of dimension
\[
\frac{(n-1)\phi(n)((n-1)\phi(n)+1)}{2},
\]
since it parametrizes abelian varieties of dimension $(n-1)\phi(n)$. 
\end{defn}

Now the Shimura datum $(\GSp(V, \Psi), \mf{H})$ certainly defines a Shimura variety to which $\mc{M}_{0,2n}$ maps; we can describe its $\mb{C}$-points as follows. The integral structure on $V$ defines a maximal compact subgroup $K\subset \GSp(V, \Psi)(\mb{A})$, and then we have 
\[
\Sh(\GSp(V, \Psi), \mf{H})(\mb{C})=\GSp(V, \Psi)(\mb{Q})\backslash \mf{H}\times \GSp(V, \Psi)(\mb{A}_f)/K;
\]
here $\mb{A}$ (respectively $\mb{A}_f$) denotes the ring of (respectively finite) adeles.  However, as we shall see presently $\mc{M}_{0,2n}$ lands inside a smaller Shimura subvariety.
\begin{defn}\label{defn:shimura}\hfill
\begin{enumerate}
    \item 
For a $\mb{Q}$-algebraic subgroup $H\subset \GSp(V, \Psi)$, define 
\[
\mf{H}_{H}\defeq \{h\in \mb{S}\rightarrow \GSp(V_{\mb{R}}, \Psi)|h \ \text{factors through} \ H_{\mb{R}}\}.
\]
\item Recall that there is a $\mb{Q}(\zeta)$-action on $V$. Define the algebraic group
\[
G\defeq \GL_{\mb{Q}(\zeta)}(V)\cap \GSp(V,\Psi);
\]
here $\GL_{\mb{Q}(\zeta)}(V)$ denotes the elements of $\GL(V)$ commuting with the action of $\mb{Q}(\zeta)$ on $V$.
\item Let $\Sh$ denote the Shimura variety associated to the Shimura datum $(G, Y_G)$;  by construction this is  a subvariety of the Shimura variety associated to the Shimura datum $(\GSp(V, \Psi), \mf{H})$.
\end{enumerate}

\end{defn}
\begin{prop}\hfill
\begin{enumerate}
\item 
The real points of the group $H$ are given by 
\[
G_{\mb{R}}\cong \prod_{r\in (\mb{Z}/n)^{\times}}U(V[r]\otimes_{\mb{Q}(\zeta)}\mb{C}),
\]
where, on the right hand side, in the $r^{\rm{th}}$ factor of the product the embedding $\mb{Q}(\zeta)\hookrightarrow\mb{C}$ is the one sending  $\zeta$ to $\zeta^r$.
\item 
The dimension of $\Sh$ is 
\begin{equation}\label{eqn:dimshimura}
\frac{1}{2}\sum_{r\in (\mb{Z}/n)^{\times}} (2r-1)(2n-1-2r).
\end{equation}
Here the sum is over  representatives between $1$ and $n$ of the elements of  $(\mb{Z}/n)^{\times}$.

\end{enumerate}
\end{prop}
\begin{proof}
The first part follows from \cite[Remark 4.6]{moonen}. For the second part, it suffices to find the signature of the pairing on each of the subspaces $V[r]$, since the hermitian symmetric domain for the unitary group $U(a,b)$ has dimension $ab$. The signatures, or equivalently the Hodge numbers, are given by   Proposition \ref{prop:hodgenumberscurve}, and (\ref{eqn:dimshimura}) follows immediately.
\end{proof}
The Prym construction therefore gives us a map 
\[
P: \mc{M}_{0,2n}\rightarrow \Sh.
\]
We have the following result, which is analogous to the classical result that the Torelli map is an embedding, although we only require an infinitesimal version of this.
\begin{lem}
The derivative of  $P$ is injective.
\end{lem}
\begin{proof}
We will show equivalently that, for each $x\in \mc{M}_{0,2n}$, the codifferential map 
\[
P^*: T_{P(x)}^*\Sh \rightarrow T_x^* \mc{M}_{0,2n}
\]
is surjective. Here for a variety $X$ and a point $x\in X$ we denote by $T^*_xX$ the cotangent space to $X$ at $x$. 

First we identify the source and target of $P^*$ in terms of the geometric structures at hand.
\begin{claim}\label{claim:cotangent}
We have the following identifications of the cotangent spaces:
\begin{align}
    T^*_{P(x)}\Sh &\cong \bigoplus_{\substack{r\in (\mb{Z}/n)^{\times}\\ r<n/2}} H^0(C, \Omega^1_C)[r]\otimes H^0(C, \Omega^1_C)[n-r], \label{eqn:cotangentshimura}\\
    T_x^* \mc{M}_{0,2n} &\cong H^0(C, (\Omega^1_C)^{\otimes 2})_{\inv} \label{eqn:cotangentmoduli}.
\end{align}
Here $H^0(C, \Omega^1_C)[r]$ denotes the $\zeta^r$ eigenspace of $H^0(C, \Omega^1_C)$, and the subscript $\inv$ in (\ref{eqn:cotangentmoduli}) denotes the invariant part of the $\mb{Z}/n$-action. 

Furthermore, under these identifications, the restriction of  $P^*$ on each of the factors is  in (\ref{eqn:cotangentshimura}) is given by the cup product map
\[
\cup: H^0(C, \Omega^1_C)[r]\otimes H^0(C, \Omega^1_C)[n-r]\rightarrow H^0(C, (\Omega^1_C)^{\otimes 2})_{\inv}
\]
    
\end{claim}
Let us first see how to conclude the proof of this lemma, assuming this claim. Note that the dimension of $H^0(C, (\Omega^1_C)^{\otimes 2}_{\inv}$ is $2n-3$, since it has to be the dimension of $\mc{M}_{0,2n}$.  We will now show that  the restriction (which we continue to denote by $P^*$)
\begin{equation}\label{eqn:resp*}
P^*: H^0(C,\Omega^1_C)[1]\otimes H^0(C,\Omega^1_C)[n-1]\rightarrow H^0(C, (\Omega^1_C)^{\otimes 2})_{\inv}
\end{equation}
is in fact an isomorphism; certainly the dimensions of the source and target agree, and so it suffices to show this map is injective. But this is clear since the space $H^0(C,\Omega^1_C)[1]$ is one-dimensional and spanned by $\omega$, say, and so anything in the kernel of the map (\ref{eqn:resp*}) takes the form $\omega\otimes \eta$ for some $\eta\in H^0(C, \Omega^1_C)[n-1]$. On the other hand $\omega$ and $\eta$ are non-zero 1-forms on $C$, and the quadratic differential obtained by multiplying them together is certainly non-zero. This shows that the map (\ref{eqn:resp*}) is injective, and therefore it is an isomorphism, as required.
Therefore it suffices to prove Claim \ref{claim:cotangent}:
\begin{proof}[Proof of Claim]
By construction, we have an embedding 
\[
\Sh \rightarrow \Sh(\GSp(V, \Psi), \mf{H}),
\]
where the right hand side denotes the Shimura variety attached to the Shimura datum $(\GSp(V, \Psi), \mf{H})$. The latter is the moduli space of abelian varieties of dimension $(n-1)\phi(n)$ equipped with a polarization of the fixed  type specified by the polarization on the Prym variety. Therefore the tangent  space to $\Sh(\GSp(V, \Psi), \mf{H})$ at  $P(x)$ is given by 
\begin{equation}\label{eqn:deformav}
\Sym^2(t_{\Prym})\subset t_{\Prym}\otimes t_{\Prym^{\vee}},
\end{equation}
where for an abelian variety $A$ we denote by $t_A$ the tangent space at the origin, and $A^{\vee}$ its  dual abelian variety. On the left hand side of the  above we have also  made the identification 
\[
t_{\Prym}\cong  t_{\Prym^{\vee}}
\]
using the polarization on $\Prym$. Note that in (\ref{eqn:deformav}) the right hand side is the deformation space of $\Prym$ with no reference to polarizations. By Corollary \ref{cor:prymhodge}, we may make the identification 
\[
t_{\Prym}\cong \bigoplus_{i\in (\mb{Z}/n)^{\times}} H^1(C, \mc{O}_C)[i];
\]
for convenience we denote the right hand side of this identification by $H^1(C, \mc{O}_C)_{\prim}$;  similarly we define $H^0(C, \Omega^1_C)_{\prim}$ to be the sum of the eigenspaces of $H^0(C, \Omega^1_C)$ with eigenvalues primitive $n$th roots of unity. 

On the other hand, as mentioned above, the right hand side of (\ref{eqn:deformav}) is the deformation space of the abelian variety $\Prym$, and hence there is a  Kodaira-Spencer map 
\begin{equation}\label{eqn:ks}
\KS: t_{\Prym}\otimes t_{\Prym^{\vee}}\rightarrow \Homo(H^0(C, \Omega^1_C)_{\prim}, H^1(C, \mc{O}_C)_{\prim}),
\end{equation}
which is the natural isomorphism once we make the identifications
\[
t_{\Prym}\otimes t_{\Prym^{\vee}}\cong H^1(C, \mc{O}_C)_{\prim}^{\otimes 2},
\]
and 
\[
H^1(C, \mc{O})_{\prim}\cong H^0(C, \Omega^1_C)_{\prim}^{\vee}, 
\]
the latter of which is induced by Serre duality.

By definition,  $\Sh$ is contained in the locus of $\Sh(\GSp(V, \Psi), \mf{H})$ where the Hodge structure $H^1$ admits a $\mb{Q}(\zeta)$-action  and a splitting
\[
H^1\otimes_{\mb{Q}} \mb{Q}(\zeta)\cong \bigoplus_{r\in (\mb{Z}/n)^{\times}}H^1[r]
\]
with prescribed Hodge numbers. Therefore the Kodaira-Spencer map (ref{eqn:ks}) restricted to $\Sh$ must preserve the different eigenspaces. In other words, we have 
\begin{equation}\label{eqn:ksshimura}
\KS|_{\Sh}: T_{P(x)}\Sh\rightarrow \bigoplus_{r\in (\mb{Z}/n)^{\times}}\Homo(H^0(C, \Omega^1_C)[r], H^1(C, \mc{O}_C)[r]);
\end{equation}
now since $\KS$ itself is an isomorphism, $\KS|_{\Sh}$ is injective at least; on the other hand, deformations in $\Sh$ are also required to  preserve the polarization, which means  further that
\begin{equation}\label{eqn:ksshimurapol}
\KS|_{\Sh}: T_{P(x)}\Sh\rightarrow \bigoplus_{\substack{r\in (\mb{Z}/n)^{\times}\\ r<n/2}}\Homo(H^0(C, \Omega^1_C)[r], H^1(C, \mc{O}_C)[r]).
\end{equation}

Since the dimensions of the two sides of (\ref{eqn:ksshimurapol}) now agree, and the map is injective,  it must in fact be an isomorphsm.

Again using Serre duality, for each $r=1, \cdots, n$, we have 
\[
H^0(C, \Omega^1_C)[r]\cong H^1(C, \mc{O}_C)[n-r]^{\vee},
\]
and using the remark above  we may rewrite (\ref{eqn:ksshimurapol}) as 
\begin{align} 
T_{P(x)}\Sh &\cong \bigoplus_{\substack{r\in (\mb{Z}/n)^{\times}\\ r<n/2}} (H^0(C, \Omega^1_C)[r])^{\vee}\otimes H^1(C, \mc{O}_C)[r] \\
           &\cong \bigoplus_{\substack{r\in (\mb{Z}/n)^{\times}\\ r<n/2}} H^1(C, \mc{O}_C)[n-r]\otimes H^1(C, \mc{O}_C)[r].
\end{align}
Dualizing, we therefore deduce  
\[
T^*_{P(x)}\Sh \cong \bigoplus_{\substack{r\in (\mb{Z}/n)^{\times}\\r<n/2}}   H^0(C, \Omega^1_C)[r]\otimes H^0(C, \Omega^1_C)[n-r],
\]
which  proves (\ref{eqn:cotangentshimura}), as claimed.

The identification (\ref{eqn:cotangentmoduli}) is well known: see for example  \cite[Proposition 4.1]{langeprym}; furthermore, \cite[Proposition 4.1]{langeprym} also shows that the codifferential of the map 
\[
\mc{M}_{0,2n}\rightarrow \Sh(\GSp(V, \Psi), \mf{H}),
\]
namely 
\[
P^*:\Sym^2(H^0(C, \Omega^1_C)_{\prim})\rightarrow H^0(C, (\Omega^1_C)^{\otimes 2})_{\inv}
\]
is the cup product map followed by projection onto the invariant factor, and therefore the same is true for the codifferential of the map
\[
\mc{M}_{0,2n}\rightarrow \Sh,
\]
as claimed. This concludes the proofs of all the statements in Claim \ref{claim:cotangent}.
\end{proof}
\end{proof}
This gives immediately the following 
\begin{cor}\label{cor:embedding}
The dimension of the image  $P(\mc{M}_{0,2n})$ inside $\Sh$ is $2n-3$.
\end{cor}


\subsection{Special subvarieties}\label{section:specialsub}
Now that we have defined the relevant PEL type Shimura variety $\Sh$, we can rephrase Theorem \ref{thm:splittingcm} in terms of special subvarieties of $\Sh$.

For $A$ one of the finite number of isogeny representatives of abelian varieties as in Proposition~\ref{man}, we denote by $\Sh_A$ the sub-Shimura variety of $\Sh$ classifying abelian varieties parametrized by $\Sh$ that split, as a product, of $A$ and some other abelian variety. Then $\Sh_A$ is a Shimura variety associated to group the Weil restriction, from $\mb{Q}(\zeta_n)^+$ to $\mb{Q}$, of a unitary group $G_A$. We now describe this in detail.

We now introduce some notation to describe the signatures of $G_A$ under the various embeddings $\mb{Q}(\zeta_n)^+ \hookrightarrow \mb{R}$. Indeed, $A$ is also an abelian variety with compatible $\mb{Z}/n$-action and Hodge structure; in other words, each eigenspace $H^1(A)[i]$ over $\mb{Q}(\zeta_n)$ again splits as $$H^1(A)[i] \underset{\mb{Q}(\zeta_n)}{\otimes} \mb{C} \simeq H^{1, 0}(A)[i] \oplus H^{0, 1}(A)[i].$$ (We will almost immediately argue that this splitting, once again, is defined over $\mb{Q}(\zeta_n)$.) But the Galois conjugates of $\omega$ are in distinct eigenspaces $H^1(A)[i]$ and already account for $\phi(n)$ dimensions' worth of cohomology -- which is the total dimension of $H^1(A)$ as a vector space. Hence, for each $i$, the total dimension of $H^{1, 0}(A)[i] \oplus H^{0, 1}(A)[i]$ is one, and so one space has dimension one while the other has dimension zero. 
\begin{defn} Define 
$$n_A(i) \defeq \dim H^{1, 0}(A)[i].$$ 
From the remark above the $n_A(i)$'s take values either zero or one. We will sometimes also  refer to the $n_A(i)$s as the \textit{CM type} of $A$, since they encode information  equivalent to the standard notion of CM types.
\end{defn}
In particular, note that as $\omega$ itself is in the first eigenspace and is holomorphic, we know $n_A(1) = 1$. 

Recall that $V_{\mb{Z}}$ denotes the integral lattice given by the Prym variety and that $G$ denotes the $\mb{Q}$-algebraic group associated to the Shimura variety $\Sh$. Now for each CM type $\{n_A(i)\}$, we fix a splitting of the integral Hodge structure of the form
\begin{equation}\label{eqn:integralsplitting}
V_{\mb{Z}}\cong V_A\oplus V',
\end{equation}
where both $V_A$ and $V'$ have actions by $\mb{Z}[\zeta]$, and such that $V_A$ has the CM type given by $\{n_A(i)\}$; more precisely, the $\mb{Z}[\zeta]$-action on $V_A$ allows us to define eigenspaces $V_A[i]$ as before, and we require
\[
\dim_{\mb{C}}V_A^{1,0}[i]=n_A(i)
\]
for each $i=1, \cdots, n$.
\begin{defn}
Let $G_A$ denote the $\mb{Q}$-algebraic group which is the subgroup of $G$ preserving the subspaces $V_A\otimes_{\mb{Z}} \mb{Q}$ and $V'\otimes_{\mb{Z}}\mb{Q}$.
\end{defn}

\begin{defn}\hfill
\begin{enumerate}
    \item 
For a $\mb{Q}$-subgroup $H\subset G$ such that the subspace $\mf{H}_H$ (see Definition \ref{defn:shimura}) is non-empty, let $\mf{H}_H^+$ denote a connected component of $\mf{H}_H$.  A special subvariety (for the subgroup $H$) is the image under the uniformization map 
\[
\mf{H}\times G(\mb{A}_f)/K\rightarrow G(\mb{Q})\backslash \mf{H}\times G(\mb{A}_f)/K=\Sh(\mb{C})
\]
of $\mf{H}^+_H\times \eta K$, for some element $\eta\in G(\mb{A}_f)$.
\item Let $\Sh_A$ denote the special subvariety associated to the inclusion of a connected component of $\mf{H}_{G_A}$, and the element $\eta=1$.
\end{enumerate}
\end{defn}
From Theorem \ref{thm:splittingcm} we immediately deduce the following slight refinement:
\begin{prop}\label{prop:splittingcmshimura}
For each attractor point $x$ in $\mc{M}_{0,2n}$ there exists a CM type $\{n_A(i)\}$ such that $x$ lies in a special subvariety of $\Sh$ associated to the group $G_A$. 
\end{prop}
\begin{rmk}
The element $\eta\in G(\mb{A}_f)$ is measuring  the isogeny to the reference integral Hodge splitting (\ref{eqn:integralsplitting}).
\end{rmk}

We may now describe the signatures the unitary form defining $G_A$ in terms of these numbers $n_A(i)$: under the $i^{\mathrm{th}}$ embedding of $\mb{Q}(\zeta_n)^+ \hookrightarrow \mb{R}$ induced from $\zeta_n \mapsto \zeta_n^i$, we have that the unitary form now has signature $(2i - 1 - n_A(i), 2(n-i)-2+n_A(i))$ and 
$$\dim \Sh_A = \sum_{\substack{i=1,\\ i \perp n}}^{\lfloor (n-1)/2 \rfloor} (2i-1 - n_A(i))(2(n-i)-2+n_A(i)).$$  
More interesting is the codimension of $\Sh_A$ within $\Sh$, namely

$$\codim_{\mc{X}} \Sh_A = \sum_{\substack{i=1,\\ i \perp n}}^{\lfloor (n-1)/2 \rfloor} \Big( n_A(i)(2(n-i)-1) + (1 - n_A(i))(2i-1) \Big).$$ 
In particular, as $n_1(A) = 1$, we have 
\begin{equation}\label{wow} 
\codim_{\Sh} \Sh_A \ge 2n - 3,
\end{equation}
with equality if and only if there is only one value of $i$ in the sum, i.e. $\phi(n) = 2$. On the other hand, $\dim \mc{M}_{0, 2n} = 2n-3$, so unless equality holds in~\eqref{wow}, one generically expects $\mc{M}_{0, 2n}$ and (Hecke translates of) $\Sh_A$ to not intersect within $\Sh$ purely for dimensional reasons. The Zilber-Pink conjecture of transcendental number theory makes precise this expectation and we recall it shortly -- but first we allow a brief digression on the particular cases of $\phi(n) = 2$ where this dimensional expectation does not hold.

\begin{ex}
We give an example to illustrate the numerology that is at play here, for $n=5$, which is the smallest value for which the Dolgachev Calabi-Yau varieties give  counterexamples to the Attractor Conjecture. In this case, since $5$ is a prime number, the Prym variety is simply the Jacobian of $C$. Let us also fix a CM type, say $n_A(1)=n_A(2)=1$; then we write schematically the splitting of the Hodge structure as
\begin{equation}
    \begin{pmatrix}
    1 & 7\\
    3 & 5\\
    5 & 5\\
    7 & 1
    \end{pmatrix}
    = \begin{pmatrix}
    1 & 0\\
    1& 0\\
    0 & 1\\
    0 & 1\\
    \end{pmatrix}+
    \begin{pmatrix}
    0 & 7\\
    2 & 5\\
    5 & 2\\
    7 & 0 \\
    \end{pmatrix}.
\end{equation}
In the equation above we follow Notation \ref{notation:hodge}, and on the right hand side we have written the dimension matrices of the summands of this splitting.
\end{ex}
\subsection{A brief digression: the arithmetic cases of $n = 3, 4, 6$ and connections to tilings of the sphere}
In this subsection we observe that the Attractor Conjecture works remarkably well in the cases when the Calabi-Yau moduli space does happen to be a Shimura variety, and point out a connection to the  tilings of the sphere by polygons due to Engel-Smillie \cite{phil}. 
\begin{prop}
For $n=3, 4, 6$, there is a bijection between attractor points and tilings of the sphere by triangles, squares and hexagons, respectively.
\end{prop}

\begin{proof}
After quotienting by the appropriate arithmetic group, the Schwarz map $\pi$ considered in Section \ref{subsection:dense} coincides precisely with  the map denoted by $D$ in \cite[Proof of Proposition 2.5, p.7]{phil}. Furthermore, the integral points in the image of $D$ correspond to tilings  of the sphere, as required.
\end{proof}

In fact, we mention an intriguing question for the interested reader: Engel-Smillie in the above paper study a very precise generating function of the attractor points in the arithmetic $n = 3, 4, 6$ cases in terms of a mock modular form; it is reasonable to ask if there may ask any analogous (but presumably more complicated) behavior in the nonarithmetic cases. 

\section{Unlikely intersection}

We now recall the Zilber-Pink conjecture:

\begin{conj}[{\cite[Conjecture~1.3]{pink}}] Given a subvariety $Y \subset \mc{X}$ of a Shimura variety and a countable collection of special Shimura subvarieties $\{X_{\alpha}\}$ of codimension greater than $\dim Y$, if $$\bigcup_{\alpha} Y \cap \mc{X}_{\alpha} \subset Y$$ is Zariski-dense, then $Y \subset \mc{X}'$ is contained within some proper special Shimura subvariety $\mc{X}' \subset \mc{X}$. \end{conj}

We will  now argue that for $\mc{M}_{0, 2n}$   is not contained in any special Shimura subvariety of $\Sh$. Once again, these arguments hold for all $n \ge 2$; the only place where the $\phi(n) > 2$ condition enters is to force the dimensional inequality. 

Recall that we have the group $G=\GL_{\mb{Q}(\zeta)}(V)\cap \GSp(V, \Psi)$ from Definition \ref{defn:shimura}. Let $G'$ denote the $\mb{Q}$-algebraic group, obtained via restriction of scalars from $\mb{Q}(\zeta)^{+}$, whose $\mb{Q}(\zeta)$-points are given by 
\[
G'(\mb{Q}(\zeta))=\prod_{r\in (\mb{Z}/n\mb{Z})^{\times}} SU(V[r]).
\]

\begin{prop}\label{prop:monodromy}
Let $\mc{G}$ denote the $\mb{Q}$-Zariski closure of the fundamental group of $\mc{M}_{0, 2n}$ acting on $H^1(C)$. Then $\mc{G}$ contains the group $G'$ above.
\end{prop}

Before turning to the argument for this proposition, we note that this requirement is exactly the hypothesis necessary to apply Zilber-Pink to conclude Theorem~\ref{thm:main} in the $\phi(n)>2$ case of unlikely intersection. Indeed, the special Shimura subvarieties of $\mc{X}$ correspond to subgroups of $G$ 
and being contained within some special Shimura subvariety would imply a corresponding restriction on the Zariski-closure of the monodromy group. It hence remains to establish the above proposition.

\begin{proof} This is essentially due to Deligne-Mostow \cite{delignemostow}, but we will use the version stated by Looijenga in his review of Deligne-Mostow, which we now describe. For brevity, in the following, for $r\in (\mb{Z}/n)^{\times}$ we denote by $V[r]$ the $\mb{Q}(\zeta)$-vector space $H^1(C, \mb{Q}(\zeta))[r]$, and by $V$ the  $\mb{Q}$-subspace of $H^1(C, \mb{Q})$ such that 
\[
V\otimes \mb{Q}(\zeta)=\bigoplus_{r\in (\mb{Z}/n)^{\times}}V[r].
\]
 
As in \cite[Proof of Theorem 4.3]{looijenga}, we have a decomposition 
\[
\mc{G}(\mb{C})=\prod_{r\in (\mb{Z}/n)^{\times}} \mc{G}_r(\mb{C})
\]
with $\mc{G}_r(\mb{C})\subset \GL(V[r])(\mb{C})$. Since this was stated without proof in loc.cit., we provide an explanation here, although no doubt it is well known to experts. The reason is that the action of the fundamental group $\Gamma$ preserves each direct summand $V[r]$, and therefore the $\mb{Q}$-algebraic group $\mc{G}$ commutes with the action of $\mb{Q}(\zeta)$ on $H^1(C, \mb{Q})$. Therefore the base change $\mc{G}\otimes \mb{Q}(\zeta)$ commutes with the action of $\mb{Q}(\zeta)\otimes_{\mb{Q}}\mb{Q}(\zeta)$; the latter can be written as
\[
\bigoplus_{r\in (\mb{Z}/n)^{\times}}\mb{Q}(\zeta),
\]
and in particular contains idempotents $\epsilon_r$, the projection onto the $r$th factor,  for each $r\in (\mb{Z}/n)^{\times}$. Therefore we have a decomposition
\[
\mc{G}\otimes \mb{Q}(\zeta)=\prod_{r\in (\mb{Z}/n)^{\times}} \mc{G}_r,
\]
as claimed.

Furthermore, again according to \cite[Proof of Theorem 4.3]{looijenga}, $\mc{G}_r(\mb{C})$ contains the special unitary group of $V[r]$. Therefore the Zariski closure contains $G'$, as required.
\end{proof}
We now put all the ingredients together to conclude the proof of our main result.
\begin{proof}[Proof of Theorem \ref{thm:main}]
Suppose that $n\neq 3,4,6$. We assume for the sake of contradiction that attractor points are defined over $\overline{\mb{Q}}$. By Section \ref{subsection:dense} the attractor points are Zariski dense in moduli space. Now recall that we have the Prym map
\[
P:\mc{M}_{CY}\rightarrow \Sh,
\]
where $\Sh$ denotes the Shimura variety from Definition \ref{defn:shimura}, and  consider its image   $P(\mc{M}_{CY})$ inside the Shimura variety $\Sh$. That this subvariety  has dimension $2n-3$, i.e. the same dimension as $\mc{M}_{CY}$, is precisely   Corollary \ref{cor:embedding}. 

On the other hand, by Proposition \ref{prop:splittingcmshimura}, each attractor point lies in a sub-Shimura variety, whose codimension inside $\Sh$ is strictly greater than $2n-3$ by the discussion in Section  \ref{section:specialsub}.  Therefore by the Zilber-Pink conjecture, the variety $P(\mc{M}_{CY})$ must be contained in some proper special subvariety of $\Sh$, which is impossible by the monodromy computation in Proposition \ref{prop:monodromy}. Therefore the attractor points cannot be defined over $\overline{\mb{Q}}$, as required.
\end{proof}

\printbibliography[]

\end{document}